\documentclass[a4paper,onecolumn,11pt,accepted=2025-10-20]{quantumarticle}
\pdfoutput=1
\usepackage[utf8]{inputenc}
\usepackage[english]{babel}
\usepackage[T1]{fontenc}
\usepackage{amsmath}
\usepackage{hyperref}

\usepackage{tikz}
\usepackage{lipsum}

\usepackage{style}

\begin{document}

\title{Error Bounds for Open Quantum Systems with Harmonic Bosonic Bath}

\author{Kaizhao Liu}
\email{mrzt@mit.edu}
\homepage{https://drzfct.github.io/}
\affiliation{Department of Mathematics, Massachusetts Institute of Technology, 77 Massachusetts Ave, Cambridge MA 02139, USA}
\orcid{0009-0008-7753-0030}
\author{Jianfeng Lu}
\email{jianfeng@math.duke.edu}
\affiliation{Department of Mathematics, Department of Physics,
  Department of Chemistry, Duke University, Box 90320, Durham NC 27708, USA}
\homepage{https://sites.math.duke.edu/~jianfeng/}
\orcid{0000-0001-6255-5165}

\maketitle

\begin{abstract}
  We investigate the dependence of physical observables of open quantum systems with bosonic bath on the bath correlation function. 
  We provide an error estimate of the difference in physical observables induced by the variation of the bath correlation function, based on diagrammatic and combinatorial arguments.
  This gives a mathematically rigorous justification of the result in [Mascherpa et al, Phys Rev Lett 2017].
\end{abstract}

\section{Introduction}

Realistic quantum systems of interest are often coupled with an environment, which interferes with the system to a non-negligible extent. This leads to the study of open quantum systems. The theory of open quantum systems has wide applications, including quantum thermodynamics \cite{Esposito2009} and quantum information science \cite{Shor1995}. Under Markovian approximation, the evolution of open quantum system can be described by Lindblad equations \cite{Lindblad1976}. The Markovian approximation however breaks down for open systems with stronger coupling, which is often the case in practice \cite{Breuer2007}.  

In this work, we consider non-Markovian open quantum systems with harmonic bosonic bath, which means that the system is coupled with an environment modelled as harmonic bosonic modes.
While the combined evolution of the system and the bosonic environment is unitary, integrating out the bath degrees of freedom leads to a non-unitary and non-Markovian reduced dynamics of the system.
In the reduced dynamics, the influence of the bath degree of freedom is captured by the bath correlation function $B(\cdot,\cdot)$ \cite{FEYNMAN1963theory} (see \eqref{eq:bathcorrelation} for more details).

In practice, the bath correlation function may contain error when it is obtained through experiment measurements \cite{mascherpa2017open}. In simulation studies, one might also produce error in the bath correlation function due to numerical truncations or approximations made in model reductions \cite{Tanimura1989,meier1999non}. Therefore, it is important to investigate the perturbation of the system due to error in the bath correlation function. This question was previously considered in \cite{mascherpa2017open} for spin-boson systems. They obtained the following error bound of the expectation for a given observable on the system $O_s$:
\begin{equation}\label{eq:mascherpa bound}
     |\Delta\langle O_s(t) \rangle|\leqslant \|O_s\|\left[\exp(4\int_0^{t}\int_0^{s_2}|\Delta B(s_1,s_2)|\dd s_1 \dd s_2)-1\right].
  \end{equation}
  
However, this result is obtained using the coherent space path integral, the validity of which is questioned in the literature, as it is known that it might sometimes lead to incorrect results \cite{kochetov2019comment,wilson2011breakdown}. In fact, while the analysis in \cite{mascherpa2017open} uses the coherent state path integral introduced in \cite{kordas2014coherent}, debates have been held on whether the approach is consistent \cite{kochetov2019comment}.
Note that this bound only depends on the perturbation of the bath correlation function $\Delta B$, while not on the original bath correlation function, which is somewhat surprising. Indeed, the bound does not have a prefactor that is growing exponentially (or faster) in time, which one might expect following usual differential inequalities such as Gronwall's inequality. As this error bound provided in \cite{mascherpa2017open} is quite useful and also somewhat surprising from a usual mathematical point of view, it is of great interest to provide a mathematically rigorous understanding of the result, which is the main motivation of the present work. 

Our main result (Theorem~\ref{thm:bound}) is a rigorous error bound for the physical observable due to the perturbation of the bosonic bath in terms of its bath correlation function. In contrast to \cite{mascherpa2017open}, our result is not limited to the spin-boson model. Our proof is based on diagrammatic expansions and combinatorial arguments. In particular, it does not involve differential inequalities, which explain the lack of an exponential growing prefactor in the final error estimate. Instead, in the proof, we establish a combinatorial identity comparing two diagrammatic expansions (Lemma~\ref{lem:identity}), which might be of independent interest.

\smallskip 

The remainder of this paper is organized as follows. 
In Section \ref{sec:pre}, we introduce the formulation and assumptions of the open quantum system under consideration. 
In Section \ref{sec:bound}, we present and prove our main results. 
In Section \ref{sec:sb} we demonstrate the application of our result to the spin-boson system, recovering the main result \eqref{eq:mascherpa bound} in \cite{mascherpa2017open}.
Finally, some concluding remarks are given in Section \ref{sec:conclusion}.

\section{Open Quantum Systems and Diagrammatic Representations}\label{sec:pre}
\subsection{Preliminary: Dyson series expansion and Keldysh contours}

Before considering the open quantum system in question, we first introduce the time-dependent perturbation theory and the associated Dyson series, following \cite{Cai2020}. 
Consider the von Neumann equation for quantum evolution (of a closed system)
\begin{equation} \label{eq:vonNeumann}
\ii \frac{\dd \rho}{\dd t} = [H, \rho],
\end{equation}
where $\rho(t)$ is the density matrix at time $t$, and $H$ is the Schr\"odinger picture Hamiltonian with the form
\begin{equation}\label{eq:total hamiltonian with interaction}
H = H_0 + W.
\end{equation}
Here, $H_0$ is the unperturbed Hamiltonian and $W$ is viewed as a perturbation.
Following the convention, for any Hermitian operator $A$, we define $\langle A\rangle = \tr(\rho(0) A)$. We are interested in the evolution of the expectation for a given observable $O$ at time $t$, defined by
\begin{equation} \label{eq:O(t)}
\langle O(t) \rangle = \tr(O \rho(t))
  = \tr(O \ee^{-\ii t H} \rho(0) \ee^{\ii t H})
  = \langle \ee^{\ii t H} O \ee^{-\ii t H} \rangle.
\end{equation}
Using standard time dependent perturbation theory, the unitary group $\ee^{-\ii t H}$ generated by $H$ can be represented using a Dyson series expansion (see e.g. \cite{sakurai}):
\begin{equation} \label{eq:Dyson}
\begin{aligned}
    \ee^{-\ii t H}& = \sum_{n=0}^{+\infty}
  \int_{t > t_n > \cdots > t_1 > 0} (-\ii)^n \times \\
    &\quad \times \ee^{-\ii(t-t_n)H_0} W \ee^{-\ii(t_n - t_{n-1})H_0} W \cdots
  W \ee^{-\ii(t_2 - t_1)H_0} W \ee^{-\ii t_1 H_0}
    \,\dd t_1 \cdots \,\dd t_n,
\end{aligned}
\end{equation}
where the integral should be interpreted as
\begin{equation*}
\int_{t > t_n > \cdots > t_1 > 0}
  \,\dd t_1 \cdots \,\dd t_n =
\int_0^t \int_0^{t_n} \cdots \int_0^{t_2}
  \,\dd t_1 \cdots \,\dd t_{n-1} \,\dd t_n.
\end{equation*}
Inserting the Dyson series \eqref{eq:Dyson} into \eqref{eq:O(t)}, one obtains
\begin{equation} \label{eq:O}
\begin{split}
\langle O(t) \rangle &= \sum_{n=0}^{+\infty} \sum_{n'=0}^{+\infty}
  \int_{t > t_n > \cdots > t_1 > 0}
    \int_{t > t'_{n'} > \cdots > t'_1 > 0} (-\ii)^n \ii^{n'}\times \\
  & \quad \times\langle \ee^{\ii t'_1 H_0} W \ee^{\ii (t'_2 - t'_1) H_0}  W \cdots
      W \ee^{\ii (t'_{n'} - t'_{n'-1}) H_0} W \ee^{\ii (t-t'_{n'}) H_0}
      O \times {}\\
   \phantom{\langle} & \quad \times\ee^{-\ii (t-t_n) H_0} W
    \ee^{-\ii (t_n - t_{n-1}) H_0} W
    \cdots W \ee^{-\ii (t_2 - t_1) H_0} W \ee^{-\ii t_1 H_0} \rangle
    \,\dd t_1' \cdots \,\dd t_n' \,\dd t_1 \cdots \,\dd t_n.
\end{split}
\end{equation}
For notational simplicity, the above integral is often denoted by the Keldysh contour \cite{kamenev2023field} plotted in Figure \ref{fig:Keldysh}.
The Keldysh contour should be read following the arrows in the diagram, and therefore has a forward (upper) branch and a backward (lower) branch. 
The symbols are interpreted as follows:
\begin{itemize}
\item Each line segment connecting two adjacent time points labeled by $t_{\mathrm{s}}$ and $t_{\mathrm{f}}$ means a propagator $\ee^{-\ii(t_{\mathrm{f}} - t_{\mathrm{s}}) H_0}$. 
On the forward branch, $t_{\mathrm{f}} > t_{\mathrm{s}}$, while on the backward branch, $t_{\mathrm{f}} < t_{\mathrm{s}}$.
\item Each black dot introduces a perturbation operator $\pm\ii W$, where we take the minus sign on the forward branch, and the plus sign on the backward branch. 
At the same time, every black dot also represents an integral with respect to the label, whose range is from $0$ to the adjacent label to its right.
\item The cross sign at time $t$ means the observable in the Schr\"odinger picture.
\end{itemize}
Note that according to the above interpretation, two Keldysh contours differ only when at least one of the values of $n$, $n'$ and $t$ is different, while the positions of the labels on each branch do not matter. 
Thus, by taking the expectation $\langle \cdot \rangle$ of this ``contour'', we obtain the summand in \eqref{eq:O}. 
Therefore $\langle O(t) \rangle$ can be understood as the sum of the expectations of all possible Keldysh contours.
\begin{figure}[!ht]
\centering
\includegraphics[width=0.45\textwidth]{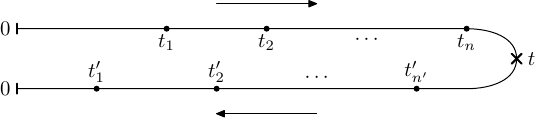}
\caption{Keldysh contour}
\label{fig:Keldysh}
\end{figure}

Such an interpretation also shows that we do not need to distinguish the forward and backward branches when writing down the integrals. 
In fact, we can reformulate the series as
\begin{equation} \label{eq:observable}
\begin{split}
\langle O(t) \rangle &= \sum_{m=0}^{+\infty}
  \int_{2t > s_m > \cdots > s_1 > 0} (-1)^{\#\{\bs < t\}} \ii^m \times {} \\
& \quad \times
  \left\langle G^{(0)} (2t, s_m) W G^{(0)} (s_m, s_{m-1}) W
  \cdots W G^{(0)} (s_2, s_1) W G^{(0)} (s_1, 0) \right\rangle
  \,\dd s_1 \cdots \,\dd s_m,
\end{split}
\end{equation}
where we use $\bs$ as a short-hand for the decreasing sequence $(s_m, \cdots, s_1)$ and use $\#\{\bs < t\}$ to denote the number of elements in $\bs$ which are less than $t$, \textit{i.e.}, the number of $s_k$ on the forward branch of the Keldysh contour. 
For a given $t$, the propagator $G^{(0)} $ is defined as
\begin{equation} \label{eq:G0}
G^{(0)} (\Sf, \Si) = \left\{ \begin{array}{ll}
  \ee^{-\ii (\Sf - \Si) H_0}, & \text{if } \Si \leqslant \Sf < t, \\
  \ee^{\ii (\Sf - \Si) H_0}, & \text{if } t \leqslant \Si \leqslant \Sf, \\
  \ee^{\ii (\Sf - t) H_0} O \ee^{-\ii (t - \Si) H_0},
    & \text{if } \Si < t \leqslant \Sf.
\end{array} \right.
\end{equation}
Note that the observable $O$ is inserted into the propagator $G^{(0)} $ to keep the expression in \eqref{eq:observable} concise.
The integral \eqref{eq:observable} can also be understood diagrammatically as the ``unfolded Keldysh contour'' \cite{Cai2020}, plotted in Figure \ref{fig:UnfoldedKeldysh}. 
In order to use only a single integral in \eqref{eq:observable}, we set the range of the unfolded Keldysh contour to be $[0, 2t]$, and the mapping of time points from the unfolded Keldysh contour to the original Keldysh contour has been implied in the definition of $G^{(0)} (\cdot,\cdot)$. 
By comparing \eqref{eq:observable} with Figure \ref{fig:UnfoldedKeldysh}, one can see that $G^{(0)} (\cdot,\cdot)$ can be considered as the unperturbed propagator on the unfolded Keldysh contour, with an action of observable $O$ at time $t$.

\begin{figure}[!ht]
\centering
\begin{tikzpicture}
\draw[-] (0,0)--(2,0);
\draw[dotted] (2,0)--(5,0); 
\draw[-] (5,0)--(8,0); 
\draw[dotted] (8,0)--(11,0); 
\draw[-] (11,0)--(13,0); 
\draw plot[only marks, mark=*, mark options={color=black, scale=0.5}] coordinates {(1,0)(2,0)(5,0)(8,0)(11,0)(12,0)};
\draw[-] (0,0.1)--(0,-0.1);
\draw[-] (13,0.1)--(13,-0.1);
\draw[-] (6.4,0.1)--(6.6,-0.1);
\draw[-] (6.4,-0.1)--(6.6,0.1);
\node at (0,-0.1) [below] {$0$};
\node at (1,0) [below] {$s_1$};
\node at (2,0) [below] {$s_2$};
\node at (5,0) [below] {$s_n$};
\node at (6.5,-0.1) [below] {$t$}; 
\node at (8,0) [below] {$s_{n+1}$};
\node at (11,0) [below] {$s_{m-1}$};
\node at (12,0) [below] {$s_m$};
\node at (13,-0.1) [below] {$2t$};
\draw[->] (5, 0.5) -- (8, 0.5); 
\end{tikzpicture}
\caption{Unfolded Keldysh contour}
\label{fig:UnfoldedKeldysh}
\end{figure}

\subsection{Diagrammatic representations under Wick's Condition}
To proceed, we now assume that the von Neumann equation \eqref{eq:vonNeumann} describes an open quantum system coupled with a bath, which means that both $\rho$ and $H$ are Hermitian operators on the Hilbert space $\mc{H} = \mc{H}_s \otimes \mc{H}_b$, with $\mc{H}_s$ and $\mc{H}_b$ representing respectively the Hilbert spaces associated with the system and the bath. 
We let $H_0$ be the Hamiltonian without coupling:
\begin{equation}\label{eq:total hamiltonian}
H_0 = H_s \otimes \mathrm{Id}_b + \mathrm{Id}_s \otimes H_b,
\end{equation}
where $H_s$ and $H_b$ are respectively the uncoupled Hamiltonians for the system and the bath, and $\Id_s$ and $\Id_b$ are respectively the identity operators for the system and the bath. 
For simplicity of presentation, we assume the coupling between system and bath, denoted as $W$, takes the special form of a tensor product
\begin{equation}\label{eq:total interaction}
W = W_s \otimes W_b.
\end{equation}
In general, the coupling $W$ between system and bath is a summation of such tensor products. Our analysis can be extended to general couplings (see Remark \ref{rem:tensor}).

Moreover, we assume the initial density matrix has the separable form $\rho(0) = \rho_s \otimes \rho_b$, and we are concerned with observables acting only on the system $O = O_s \otimes \Id_b$ (recall that physically the system is the interesting part). 
With these assumptions, we can separate system and bath parts in \eqref{eq:observable}, leading to
\begin{equation} \label{eq:observable1}
\langle O(t) \rangle = \sum_{m=0}^{+\infty}
  \ii^m \int_{2t > s_m > \cdots > s_1 > 0}
    (-1)^{\#\{\bs < t\}} \tr_s(\rho_s \mathcal{U}_s (2t, \bs, 0))
    \mathcal{L}_b(\bs) \,\dd s_1 \cdots \,\dd s_m,
\end{equation}
where the integrand is separated into $\mc{U}_s $ and $\mc{L}_b$ for the system and bath parts:
\begin{align}
\label{eq:mc_U}
\begin{split}
\mc{U}_s (\Sf, \bs, \Si) &=
  \mc{U}_s (\Sf, s_m, \cdots, s_1, \Si) \\
&:= G_s (\Sf, s_m) W_s G_s (s_m, s_{m-1}) W_s
  \cdots W_s G_s (s_2, s_1) W_s G_s (s_1, \Si),
\end{split} \\
\label{eq:mc_L}
\mathcal{L}_b(\bs) &:=
  \tr_b(\rho_b G_b (2t, s_m) W_b G_b (s_m, s_{m-1}) W_b
  \cdots W_b G_b (s_2, s_1) W_b G_b (s_1, 0)),
\end{align}
where $\tr_s$ and $\tr_b$ take traces of the system and bath respectively. The propagators $G_s $ and $G_b $ are defined similarly to \eqref{eq:G0}:
\begin{equation}
  G_s (\Sf, \Si) =
  \begin{cases}
    \ee^{-\ii (\Sf - \Si) H_s},
    & \text{if } \Si \leqslant \Sf < t, \\[5pt]
    \ee^{-\ii (\Si - \Sf) H_s},
    & \text{if } t \leqslant \Si \leqslant \Sf, \\[5pt]
    \ee^{-\ii (t - \Sf) H_s} O_s \ee^{-\ii (t - \Si) H_s},
    & \text{if } \Si < t \leqslant \Sf,
  \end{cases}
\end{equation}
and
\begin{equation}
  G_b (\Sf, \Si) = \begin{cases}
    \ee^{-\ii (\Sf - \Si) H_b},
    & \text{if } \Si \leqslant \Sf < t, \\[5pt]
    \ee^{-\ii (\Si - \Sf) H_b},
    & \text{if } t \leqslant \Si \leqslant \Sf, \\[5pt]
    \ee^{-\ii (2t - \Si - \Sf) H_b},
    & \text{if } \Si < t \leqslant \Sf.
    \end{cases}
\end{equation}

Finally, we make the assumption that the bath is harmonic (Gaussian), meaning that its statistical properties are completely determined by its first and second moments. This assumption is satisfied if the initial state of the bath is Gaussian and the Gaussianity is preserved by the free evolution of the bath. In particular, this assumption holds for the spin-boson model considered in \cite{mascherpa2017open}, and details are provided in Section \ref{sec:sb}.

The harmonic assumption indicates that the contribution $\mc{L}_b$ 
can be broken up into all possible pairings, satisfying the following \emph{Wick's condition} \cite{bruus2004many,Cai2020,Cai2020b,fogedby2022field}: 
\begin{equation} \label{eq:Wick}
\mc{L}_b(s_m, \cdots, s_1) = \left\{ \begin{array}{ll}
    0, & \text{if } m \text{ is odd}, \\[5pt]
    \displaystyle \sum_{\mf{q} \in \mQ(s_m, \cdots, s_1)} \mc{L}(\mf{q}),
    & \text{if } m \text{ is even},
\end{array} \right.
\end{equation}
where the right-hand side is given by all possible ordered pairings of the time points:
\begin{align}
\label{eq:Lq}
& \mc{L}(\mf{q}) = \prod_{(\tau_1,\tau_2) \in \mf{q}} B(\tau_1, \tau_2), \\
        \begin{split}
    & \mQ(s_m, \cdots, s_1) =
      \Big\{ \{(s_{j_1}, s_{k_1}), \cdots, (s_{j_{\frac{m}{2}}}, s_{k_{\frac{m}{2}}})\} \,\Big\vert\,
      \{j_1, \cdots, j_{\frac{m}{2}}, k_1, \cdots, k_{\frac{m}{2}}\} = \{1,\cdots,m\}, \\
    & \hspace{220pt} s_{j_l} < s_{k_l} \text{ for any } l = 1,\cdots,\frac{m}{2}
    \Big\},
    \end{split}
    \end{align}
with $B(\cdot,\cdot)$ being the \emph{(unfolded) bath correlation function} \cite{Cai2020,Cai2020b}, given by
\begin{equation}\label{eq:bathcorrelation}
B(\tau_1,\tau_2)=\tr_b(\rho_bG_b(2t,\tau_2)W_bG_b(\tau_2,\tau_1)W_bG_b(\tau_1,0)).
\end{equation}
As a convention, when $m = 0$, the value of $\mc{L}(\emptyset)$ is defined as $1$. 
Here in \eqref{eq:Wick}, $\mQ(s_m, \cdots, s_1)$ is the set of all possible ordered pairings of $\{s_m, \cdots, s_1\}$. For example,
\begin{align*}
  & \mQ(s_2, s_1) = \bigl\{ \{(s_1,s_2)\} \bigr\}, \\
  & \mQ(s_4, s_3, s_2, s_1) = \bigl\{ \{(s_1,s_2), (s_3,s_4)\},
    \{(s_1,s_3),(s_2,s_4)\}, \{(s_1,s_4),(s_2,s_3)\} \bigr\}.
\end{align*}
It is convenient to represent these sets by many-body diagrams $\mQ_m$:
\begin{equation} \label{eq:all linking pair diagram example}
\begin{aligned}
& \mQ_2 = \{\begin{tikzpicture}
\draw[-] (0,0)--(0.5,0);\draw plot[only marks,mark =*, mark options={color=black, scale=0.5}]coordinates {(0,0) (0.5,0)};
\draw[-] (0,0) to[bend left] (0.5,0);
 \end{tikzpicture} \} \\
&   \mQ_4
= \{
\begin{tikzpicture}
\draw[-] (0,0)--(1.5,0);\draw plot[only marks,mark =*, mark options={color=black, scale=0.5}]coordinates {(0,0) (0.5,0) (1,0)(1.5,0)};
\draw[-] (0,0) to[bend left] (0.5,0);
\draw[-] (1,0) to[bend left] (1.5,0);
 \end{tikzpicture}
 ,\quad
 \begin{tikzpicture}
\draw[-] (0,0)--(1.5,0);\draw plot[only marks,mark =*, mark options={color=black, scale=0.5}]coordinates {(0,0) (0.5,0) (1,0)(1.5,0)};
\draw[-] (0,0) to[bend left] (1,0);
\draw[-] (0.5,0) to[bend left] (1.5,0);
 \end{tikzpicture}
 ,\quad
  \begin{tikzpicture}
\draw[-] (0,0)--(1.5,0);\draw plot[only marks,mark =*, mark options={color=black, scale=0.5}]coordinates {(0,0) (0.5,0) (1,0)(1.5,0)};
\draw[-] (0,0) to[bend left] (1.5,0);
\draw[-] (0.5,0) to[bend left] (1,0);
 \end{tikzpicture}\}.
\end{aligned}
\end{equation}

Applying Wick's condition \eqref{eq:Wick}, we get from \eqref{eq:observable1} that
\begin{equation} \label{eq:observable2}
\langle O(t) \rangle = \sum_{\substack{m=0\\[2pt] m \text{ is even}}}^{+\infty}
   \int_{2t > s_m > \cdots > s_1 > 0}
   \sum_{\mf{q} \in \mQ(\bs)} (-1)^{\#\{\bs < t\}} \ii^m \tr_s(\rho_s\mc{U}(2t, \bs, 0))
   \mc{L}(\mf{q}) \,\dd s_1 \cdots \,\dd s_m.
\end{equation}
The Wick's condition \eqref{eq:Wick} allows us to use diagrammatic notations to conveniently represent these high-dimensional integrals, as explained below. 

The integral of $$(-1)^{\#\{\bs < t\}} \ii^m \tr_s(\rho_s\mc{U}^{(0)}(\Sf, \bs, \Si)) \mc{L}(\mf{q})$$
can be represented by a diagram as in Figure \ref{fig:Dyson}, which is
interpreted by
\begin{itemize}
\item Each line segment connecting two adjacent time points labeled by $t_{\mathrm{s}}$ and $t_{\mathrm{f}}$ represents a propagator $G_s^{(0)}(t_{\mathrm{f}}, t_{\mathrm{s}})$.
\item Each black dot introduces a perturbation operator $\pm\ii W_s$, and we take the minus sign on the forward branch, and the plus sign on the backward branch. 
Here the label for time $t$, which separates the two branches of the Keldysh contour, is omitted. 
Additionally, each black dot also represents the integral with respect to the label, whose range is from $\Si$ to the adjacent label to its right.
\item The arc connecting two time points $t_{\mathrm{s}}$ and $t_{\mathrm{f}}$ stands for $B(t_{\mathrm{s}}, t_{\mathrm{f}})$.
\end{itemize}
\begin{figure}[!ht]
\centering

    \begin{tikzpicture}
\draw[-] (0,0)--(7.0,0);
\draw plot[only marks, mark=*, mark options={color=black, scale=0.5}] coordinates { (1.0,0)(2,0)(3,0)(4,0) (5,0)(6,0) };
\draw[-] (1.0,0) to[bend left] (6.0,0);
\draw[-] (3.0,0) to[bend left] (5.0,0);
\draw[-] (2.0,0) to[bend left] (4.0,0);

 \node at (0,-0.2) [below] {$\Si$};
 \node at (1,0) [below] {$s_1$};
  \node at (2,0) [below] {$s_2$};
   \node at (3,0) [below] {$s_3$};
    \node at (4,0) [below] {$s_4$};
     \node at (5,0) [below] {$s_5$};
      \node at (6,0) [below] {$s_6$};
 \node at (7,-0.2) [below] {$\Sf$};
\end{tikzpicture}

\caption{Diagrammatic representation for the integral of $(-1)^{\#\{\bs < t\}} \ii^m \tr_s(\rho_s\mc{U}^{(0)}(\Sf, \bs, \Si)) \mc{L}(\mf{q})$ when $m = 6$ and $\mf{q} =
\{(s_1,s_6), (s_2,s_4), (s_3,s_5)\}$.}
\label{fig:Dyson}
\end{figure}
\noindent Note that the branches are not explicitly labeled in Figure \ref{fig:Dyson}. The two end points $\Si$ and $\Sf$ may both locate on the forward branch or the backward branch; they may also belong to different branches. 

Such a diagrammatic representation allows us to rewrite
\eqref{eq:observable2} as
\begin{align*}
    \begin{tikzpicture}
\draw[line width=0.4mm] (0,0)--(3.0,0);
\end{tikzpicture}&=
    \begin{tikzpicture}
\draw[-] (0,0)--(3.0,0);
\draw plot[only marks, mark=*, mark options={color=black, scale=0.5}] coordinates { (1.0,0) (2,0) };
\draw[-] (1.0,0) to[bend left] (2.0,0);
\end{tikzpicture}\\&\quad+
\begin{tikzpicture}
\draw[-] (0,0)--(3.0,0);
\draw plot[only marks, mark=*, mark options={color=black, scale=0.5}] coordinates { (0.6,0) (1.2,0) (1.8,0) (2.4,0)};
\draw[-] (0.6,0) to[bend left] (1.2,0);
\draw[-] (1.8,0) to[bend left] (2.4,0);
 \end{tikzpicture}+
 \begin{tikzpicture}
\draw[-] (0,0)--(3.0,0);
\draw plot[only marks, mark=*, mark options={color=black, scale=0.5}] coordinates { (0.6,0) (1.2,0) (1.8,0) (2.4,0)};
\draw[-] (0.6,0) to[bend left] (1.8,0);
\draw[-] (1.2,0) to[bend left] (2.4,0);
 \end{tikzpicture}+
  \begin{tikzpicture}
\draw[-] (0,0)--(3.0,0);
\draw plot[only marks, mark=*, mark options={color=black, scale=0.5}] coordinates { (0.6,0) (1.2,0) (1.8,0) (2.4,0)};
\draw[-] (0.6,0) to[bend left] (2.4,0);
\draw[-] (1.2,0) to[bend left] (1.8,0);
 \end{tikzpicture}
 \\&\quad+\cdots
\end{align*}
where time points are not explicitly labeled here for simplicity.

\begin{remark} \label{rem:boson}
The assumption \eqref{eq:Wick} is a bosonic Wick's condition. 
For fermions, a power of $-1$ needs to be added into the definition of $\mc{L}(\mf{q})$ \eqref{eq:Lq} (see for example \cite{bruus2004many}). 
While for definiteness of presentation we stick to the bosonic bath in this paper, the methodology and analysis are generalizable to the fermionic case as well.
\end{remark}

\begin{remark} \label{rem:tensor}
    When the coupling $W$ is not a tensor product, our analysis is applicable to cases where a generalized form of Wick's condition holds, see for example (28) in \cite{aurell2020operator} and (10) in \cite{cirio2025input}. For  simplicity, we will not go into the details. 
\end{remark}

\medskip

Let $\|\cdot\|$ denote the operator norm on the Hilbert space associated with the system $\mathcal{H}_s$. 
Under suitable boundedness assumptions, we establish the absolute convergence of \eqref{eq:observable2}, and thus justifies the infinite series.

\begin{prop}\label{prop:ac}
    Assume the bath correlation function $B(\cdot,\cdot)$ is bounded, and the operators $O_s$ and $W_s$ are bounded on $\mathcal{H}_s$. Then the series \eqref{eq:observable2} is absolutely convergent.
\end{prop}
\begin{proof}
Note that 
\begin{displaymath}
    |\tr_s(\rho_s \mc{U}_s(2t, \bs, 0)) |\leqslant \| \mc{U}_s(2t,\bs,0) \|.
\end{displaymath}
By the definition of
$\mc{U}_s$ in \eqref{eq:mc_U} and using the fact that $\|e^{itH_s} \|=1$, we immediately have
\begin{displaymath}
\| \mc{U}_s(2t,\bs,0) \| \leqslant \|W_s\|^m \|O_s\|.
\end{displaymath}
 As $B(\cdot,\cdot)$ is bounded, there exists $C>0$ such that $|B(\cdot,\cdot)|\leqslant C$. Thus when $m$ is even,
\begin{displaymath}
|\mc{L}(\mf{q})| \leqslant C^{\frac{m}{2}},
  \qquad \forall \mf{q} \in \mQ(s_m, \cdots, s_1).
\end{displaymath}
Since the number of pair sets in $\mQ(s_m, \cdots, s_1)$ is $(m-1)!!$, we have
\begin{equation} \label{eq:abs_convergence_spin_Boson}
\begin{split}
& \sum_{\substack{m=0\\[2pt] m \text{ is even}}}^{+\infty}
   \int_{2t > s_m > \cdots > s_1 > 0} \|\mc{U}_s (2t, \bs, 0)\|
     \left| \sum_{\mf{q} \in \mQ(\bs)} \mc{L}(\mf{q}) \right|
   \,\dd s_1 \cdots \,\dd s_m \\
\leqslant {} & \sum_{\substack{m=0\\[2pt] m \text{ is even}}}^{+\infty}
   \int_{2t > s_m > \cdots > s_1 > 0}
     \|O_s\|\|W_s\|^m \cdot (m-1)!! C^{\frac{m}{2}} \,\dd s_1 \cdots \,\dd s_m \\
={} & \|O_s\| \sum_{\substack{m=0\\[2pt] m \text{ is even}}}^{+\infty}
  \frac{(2t)^m}{m!!} (C\|W_s\|^2)^{\frac{m}{2}}
=\|O_s\| \exp \left( \frac{C\|W_s\|^2(2t)^2}{2} \right).\qedhere
\end{split}
\end{equation}
\end{proof}

\section{Main Results}\label{sec:bound}
As we consider the error induced by variation of the bath correlation function, let us introduce an alternative system of the same form as \eqref{eq:total hamiltonian with interaction}, with the same $H_0$ as in \eqref{eq:total hamiltonian}, but with a different interaction term
\begin{displaymath}
\wt{W} = W_s \otimes \wt{W}_b,
\end{displaymath}
satisfying Wick's condition but with a different bath correction function $\wt{B}(\cdot,\cdot)$.
From now on we use $B$ and $\wt{B}$ to distinguish the original and  alternative systems. Moreover, we denote their difference by $\Delta B := \wt{B}-B$.
We are interested in how the observables differ
\begin{equation*}
  \Delta\langle  O(t)\rangle= \langle  O(t)\rangle_{\wt{B}} - \langle  O(t)\rangle_{B}.
\end{equation*}
Our main result is the following error bound.

\begin{theorem}\label{thm:bound}
  Assume the bath correlation functions $B(\cdot,\cdot)$ and $\wt{B}(\cdot,\cdot)$ are bounded. Assume $O_s$ and $W_s$ are bounded operators. Then
  \begin{equation}\label{eq:bound}
    |\Delta\langle O(t) \rangle|\leqslant \|O_s\|\left[\exp(\|W_s\|^2\int_0^{2t}\int_0^{s_2}|\Delta B(s_1,s_2)|\dd s_1 \dd s_2)-1\right].
  \end{equation}
\end{theorem}
We remark that the estimate \eqref{eq:bound} is slightly more general than \eqref{eq:mascherpa bound}. To make the connection with \eqref{eq:mascherpa bound} and  reference \cite{mascherpa2017open} more clear, in Section \ref{sec:sb}, we will discuss the specific example of the spin-boson model, for which \eqref{eq:bound} reduces to \eqref{eq:mascherpa bound}.

To prove the error bound, we introduce the following two results, Lemma~\ref{lem:identity} and Lemma~\ref{lem:comb}, which might be of independent interest. To state these results, we define an auxiliary operator 
\[\mathring{W}=W_s\otimes \mathrm{Id}_b ,\]
an auxiliary propagator corresponding to the total Hamiltonian $H$ of the original system $B$:
\begin{equation} \label{eq:G}
G (\Sf, \Si) = \left\{ \begin{array}{ll}
  \ee^{-\ii (\Sf - \Si) H}, & \text{if } \Si \leqslant \Sf < t, \\
  \ee^{\ii (\Sf - \Si) H}, & \text{if } t \leqslant \Si \leqslant \Sf, \\
  \ee^{\ii (\Sf - t) H} O \ee^{-\ii (t - \Si) H},
    & \text{if } \Si < t \leqslant \Sf,
\end{array} \right.
\end{equation}
and 
\begin{equation}
    \begin{aligned}
        \mathring{\mathcal{U}}(\Sf, \bs, \Si) &=
  \mathring{\mc{U}} (\Sf, s_m, \cdots, s_1, \Si) \\
&= G (\Sf, s_m) \mathring{W} G (s_m, s_{m-1}) \mathring{W}
  \cdots \mathring{W} G (s_2, s_1) \mathring{W} G (s_1, \Si).
    \end{aligned}
\end{equation}

\begin{lemma}\label{lem:identity}
Assume the bath correlation functions $B(\cdot,\cdot)$ and $\wt{B}(\cdot,\cdot)$ are bounded. Assume $O_s$ and $W_s$ are bounded operators.
Then the following equality holds:
    \begin{multline}\label{eq:identity}
         \sum_{\substack{m=0\\[2pt] m \text{ is even}}}^{+\infty} \ii^m 
   \int_{2t > s_m > \cdots > s_1 > 0} (-1)^{\#\{\bs < t\}} \tr_s(\rho_s \mathcal{U}_s (2t, \bs, 0))
      \sum_{\mf{q} \in \mQ(\bs)} \left(\mc{L}_{\wt{B}}(\mf{q})-\mc{L}_{B}(\mf{q} )\right)
   \,\prod_{i=1}^{m}\dd s_i \\
   = \sum_{\substack{m=2\\[2pt] m \text{ is even}}}^{+\infty} \ii^m \int_{2t > s_m > \cdots > s_1 > 0} (-1)^{\#\{\bs < t\}} \tr (\rho(0)\mathring{\mathcal{U} }(2t, \bs, 0))\sum_{\mf{q} \in \mQ(\bs)} \mc{L}_{\Delta{B}}(\mf{q})
   \,\prod_{i=1}^{m}\dd s_i,
    \end{multline}
    where both sides are absolutely convergent.
\end{lemma}
The absolute convergence of both sides of \eqref{eq:identity} follows directly from arguments similar to those in Proposition~\ref{prop:ac}.
Before formally proving the identity \eqref{eq:identity}, we provide a heuristic, diagrammatic argument to illustrate why it holds.

Consider the left-hand side of \eqref{eq:identity}, and reorganize it according to the number of occurrences of $\Delta B$. Let us examine the first-order terms in $\Delta B$. These terms can be represented by the following diagrams, which should be interpreted in the same way as Figure~\ref{fig:Dyson}, with the additional convention that a red arc connecting two black dots represents an insertion of $\Delta B$:
\begin{gather*}
    \begin{tikzpicture}
\draw[-] (-0.25,0)--(0.75,0);\draw plot[only marks,mark =*, mark options={color=black, scale=0.5}]coordinates {(0,0) (0.5,0)};
\draw[-,red] (0,0) to[bend left] (0.5,0);
 \end{tikzpicture},\\
\begin{tikzpicture}
\draw[-] (-0.25,0)--(1.75,0);\draw plot[only marks,mark =*, mark options={color=black, scale=0.5}]coordinates {(0,0) (0.5,0) (1,0)(1.5,0)};
\draw[-,red] (0,0) to[bend left] (0.5,0);
\draw[-] (1,0) to[bend left] (1.5,0);
 \end{tikzpicture}
 ,\quad
 \begin{tikzpicture}
\draw[-] (-0.25,0)--(1.75,0);\draw plot[only marks,mark =*, mark options={color=black, scale=0.5}]coordinates {(0,0) (0.5,0) (1,0)(1.5,0)};
\draw[-] (0,0) to[bend left] (0.5,0);
\draw[-,red] (1,0) to[bend left] (1.5,0);
 \end{tikzpicture}
 ,\quad
 \begin{tikzpicture}
\draw[-] (-0.25,0)--(1.75,0);\draw plot[only marks,mark =*, mark options={color=black, scale=0.5}]coordinates {(0,0) (0.5,0) (1,0)(1.5,0)};
\draw[-,red] (0,0) to[bend left] (1,0);
\draw[-] (0.5,0) to[bend left] (1.5,0);
 \end{tikzpicture}
 ,\quad
 \begin{tikzpicture}
\draw[-] (-0.25,0)--(1.75,0);\draw plot[only marks,mark =*, mark options={color=black, scale=0.5}]coordinates {(0,0) (0.5,0) (1,0)(1.5,0)};
\draw[-] (0,0) to[bend left] (1,0);
\draw[-,red] (0.5,0) to[bend left] (1.5,0);
 \end{tikzpicture}
 ,\quad
  \begin{tikzpicture}
\draw[-] (-0.25,0)--(1.75,0);\draw plot[only marks,mark =*, mark options={color=black, scale=0.5}]coordinates {(0,0) (0.5,0) (1,0)(1.5,0)};
\draw[-,red] (0,0) to[bend left] (1.5,0);
\draw[-] (0.5,0) to[bend left] (1,0);
 \end{tikzpicture}
 ,\quad
  \begin{tikzpicture}
\draw[-] (-0.25,0)--(1.75,0);\draw plot[only marks,mark =*, mark options={color=black, scale=0.5}]coordinates {(0,0) (0.5,0) (1,0)(1.5,0)};
\draw[-] (0,0) to[bend left] (1.5,0);
\draw[-,red] (0.5,0) to[bend left] (1,0);
 \end{tikzpicture},\\
    \cdots.
\end{gather*}
In detail, when $m=2$ and $\bs=(s_2,s_1)$,
\begin{equation*}
    \sum_{\mf{q} \in \mQ(\bs)} \left(\mc{L}_{\wt{B}}(\mf{q})-\mc{L}_{B}(\mf{q} )\right)=\Delta B(s_1,s_2),
\end{equation*}
which is represented by the first line in the above diagram. When $m=4$ and $\bs=(s_4,s_3,s_2,s_1)$,  
\begin{equation*}
\begin{aligned}
    &\sum_{\mf{q} \in \mQ(\bs)} \bigl (\mc{L}_{\wt{B}}(\mf{q}) -\mc{L}_{B}(\mf{q} )\bigl)\\
    =&\wt{B}(s_1,s_2)\wt{B}(s_3,s_4)+\wt{B}(s_1,s_3)\wt{B}(s_2,s_4)+\wt{B}(s_1,s_4)\wt{B}(s_2,s_3)\\&
    \quad-B(s_1,s_2)B(s_3,s_4)-B(s_1,s_3)B(s_2,s_4)-B(s_1,s_4)B(s_2,s_3)\\
    =& \Delta B(s_1,s_2)B(s_3,s_4)+ \Delta B(s_1,s_3)B(s_2,s_4)+\Delta B(s_1,s_4)B(s_2,s_3)\\
    &\quad +B(s_1,s_2)\Delta B(s_3,s_4)+B(s_1,s_3)\Delta B(s_2,s_4)+B(s_1,s_4)\Delta B(s_2,s_3) \\
    &\quad  +\Delta B(s_1,s_2)\Delta B(s_3,s_4)+\Delta B(s_1,s_3)\Delta B(s_2,s_4)+\Delta B(s_1,s_4)\Delta B(s_2,s_3),
\end{aligned}
\end{equation*}
where the first six terms are first-order in $\Delta B$ and are represented by the second line in the above diagram.

Now let us compare this with the first-order terms in $\Delta B$ on the right-hand side of \eqref{eq:identity}, which corresponds to the $m = 2$ term in the summation on the right-hand side of \eqref{eq:identity}.
By expanding the full propagators $G(\cdot,\cdot)$ appearing in $\mathring{\mathcal{U} }(2t, \bs, 0)$, we yield the following diagram:
\begin{align*}
    \begin{tikzpicture}
\draw[line width=0.4mm] (0,0)--(3.0,0);
\draw plot[only marks, mark=*, mark options={color=black, scale=0.5}] coordinates { (1.0,0) (2,0) };
\draw[-, red] (1.0,0) to[bend left] (2.0,0);
\end{tikzpicture}&=
    \begin{tikzpicture}
\draw[-] (0,0)--(3.0,0);
\draw plot[only marks, mark=*, mark options={color=black, scale=0.5}] coordinates { (1.0,0) (2,0) };
\draw[-, red] (1.0,0) to[bend left] (2.0,0);
\end{tikzpicture}\\&\quad+
\begin{tikzpicture}
\draw[-] (0,0)--(3.0,0);
\draw plot[only marks, mark=*, mark options={color=black, scale=0.5}] coordinates { (1.0,0) (2,0) (0.25,0) (0.75,0)};
\draw[-, red] (1.0,0) to[bend left] (2.0,0);
\draw[-] (0.25,0) to[bend left] (0.75,0);
 \end{tikzpicture}+
\begin{tikzpicture}
\draw[-] (0,0)--(3.0,0);
\draw plot[only marks, mark=*, mark options={color=black, scale=0.5}] coordinates { (1.0,0) (2,0) (1.25,0) (1.75,0)};
\draw[-, red] (1.0,0) to[bend left] (2.0,0);
\draw[-] (1.25,0) to[bend left] (1.75,0);
 \end{tikzpicture}+
 \begin{tikzpicture}
\draw[-] (0,0)--(3.0,0);
\draw plot[only marks, mark=*, mark options={color=black, scale=0.5}] coordinates { (1.0,0) (2,0) (2.25,0) (2.75,0)};
\draw[-, red] (1.0,0) to[bend left] (2.0,0);
\draw[-] (2.25,0) to[bend left] (2.75,0);
 \end{tikzpicture}
 \\&\quad+
\begin{tikzpicture}
\draw[-] (0,0)--(3.0,0);
\draw plot[only marks, mark=*, mark options={color=black, scale=0.5}] coordinates { (1.0,0) (2,0) (1.25,0) (0.75,0)};
\draw[-, red] (1.0,0) to[bend left] (2.0,0);
\draw[-] (0.75,0) to[bend left] (1.25,0);
 \end{tikzpicture}+
 \begin{tikzpicture}
\draw[-] (0,0)--(3.0,0);
\draw plot[only marks, mark=*, mark options={color=black, scale=0.5}] coordinates { (1.0,0) (2,0) (2.25,0) (1.75,0)};
\draw[-, red] (1.0,0) to[bend left] (2.0,0);
\draw[-] (1.75,0) to[bend left] (2.25,0);
 \end{tikzpicture}+
  \begin{tikzpicture}
\draw[-] (0,0)--(3.0,0);
\draw plot[only marks, mark=*, mark options={color=black, scale=0.5}] coordinates { (1.0,0) (2,0) (2.25,0) (0.75,0)};
\draw[-, red] (1.0,0) to[bend left] (2.0,0);
\draw[-] (0.75,0) to[bend left] (2.25,0);
 \end{tikzpicture}\\&\quad+\cdots
\end{align*}
where new black dots appear because when we expand $G(\cdot,\cdot)$ we introduce new time points, and black arcs connecting them arise from application of Wick's condition \eqref{eq:Wick}. 
To elaborate, the first term $m=2$ in the summation right-hand side of \eqref{eq:identity}:
\begin{align*}
    &\ii^2\int_{2t > s_2 > s_1 > 0} (-1)^{\#\{\bs < t\}} \tr (\rho(0)\mathring{\mathcal{U} }(2t, \bs, 0))\sum_{\mf{q} \in \mQ(\bs)} \mc{L}_{\Delta{B}}(\mf{q})\\
    =\,&\ii^2\int_{2t > s_2 > s_1 > 0} (-1)^{\#\{\bs < t\}} \tr (\rho(0)G (2t, s_2) \mathring{W} G (s_2, s_1) \mathring{W}
  G (s_1, 0))\Delta B(s_2,s_1)
\end{align*}
then we expand $G (2t, s_2)$, $ G (s_2, s_1) $, and $G (s_1, 0)$ as in \eqref{eq:Dyson}. Odd order terms vanish due to Wick's condition \eqref{eq:Wick}. The lowest order term introduce zero new time points, which is represented by the first line in the above diagram. The second lowest order term introduces two new time points. They can be both from $G (2t, s_2)$, both from $ G (s_2, s_1) $, both from $G (s_1, 0)$, one from $G (2t, s_2)$ the other from $ G (s_2, s_1) $, one from $G (s_1, 0)$ the other from $ G (s_2, s_1) $, or one from $G (s_1, 0)$ the other from $ G (s_2, s_1) $. Wick's condition \eqref{eq:Wick} introduces the black arc connecting the two new time labels. This yield the second and third lines in the above diagrams.

We observe that the diagrams from the left-hand side and the right-hand side are exactly the same, order by order. It is also true for higher-order terms in $\Delta B$. A rigorous proof is presented below. 

\begin{proof}
We first expand the term $\tr(\rho(0)\mathring{\mathcal{U} }(2t, \bs, 0))$ in the right-hand side of \eqref{eq:identity}, where $\bs = (s_m, s_{m-1}, \ldots, s_1)$.
 From the Dyson series expansion \eqref{eq:Dyson}, we can expand $G(2t,0)$ as the following Dyson series:
\begin{equation*}
\begin{aligned}
     G(2t,0)&=\sum_{n=0}^\infty \int_{2t>s_n>\cdots>s_1>0} \ii^n (-1)^{\#\{\bs < t\}}\times \\
     &\quad \times G^{(0)} (2t, s_n) W G^{(0)} (s_n, s_{n-1}) W
  \cdots W G^{(0)} (s_2, s_1) W G^{(0)} (s_1, 0)\prod_{i=1}^{n}\dd s_i.
\end{aligned}
\end{equation*}
We expand all $G(\cdot,\cdot)$ terms that appear in $\mathring{\mathcal{U} }(2t, \bs, 0)$, yielding
\begin{align*}
    &\tr(\rho(0)\mathring{\mathcal{U} }(2t, \bs, 0))\\=&
    \sum_{n_0=0}^\infty\cdots\sum_{n_m=0}^\infty  \int_{2t>s_m^{n_m}>\cdots>s_m^1>s_m>\cdots>s_1>s_0^{n_0}>\cdots>s_0^1>0} \ii^{\sum_{i=0}^m n_i} (-1)^{\#\{\bs' < t\}}\operatorname{tr} \Big(\rho(0)\times\\
    \quad  & \times W(s_m^{n_m})\cdots W(s_m^{1}) G^{(0)} (s_m^{1}, s_{m}) W(s_m) G^{(0)} (s_m, s_{m-1}^{n_{m-1}})W(s_{m-1}^{n_{m-1}})
  \cdots W(s_1^{1})  \times \\
   \quad & \times G^{(0)} (s_1^{1}, s_{1}) W(s_1)G^{(0)} (s_1, s_{0}^{n_{0}})W(s_{0}^{n_{0}})\cdots W(s_0^1) G^{(0)} (s_0^1, 0)\Big)\,\prod_{i=0}^m\prod_{j=1}^{n_i}\dd s_i^{j},
\end{align*}
where the new time sequence $\bs'=(s_m^{n_m},\cdots,s_m^1,s_m,s_{m-1}^{n_{m-1}},\cdots,s_1^1,s_1,s_0^{n_0},\cdots,s_0^1)$ can be represented by the following contour (red color used to indicate $\bs$): 
\begin{equation*}
\begin{tikzpicture}[scale=13/15]
\draw[-] (0,0)--(1,0);
\draw[dotted] (1,0)--(2,0);
\draw[-] (2,0)--(4,0);
\draw[dotted] (4,0)--(5,0);
\draw[-] (5,0)--(6,0);
\draw[dotted] (6,0)--(9,0);
\draw[-] (9,0)--(10,0);
\draw[dotted] (10,0)--(11,0);
\draw[-] (11,0)--(13,0);
\draw[dotted] (13,0)--(14,0);
\draw[-] (14,0)--(15,0);
\draw plot[only marks, mark=*, mark options={color=black, scale=0.5}] coordinates {(1,0)(2,0)(4,0)(5,0)(10,0)(11,0)(13,0)(14,0)};
\draw plot[only marks, mark=*, mark options={color=red, scale=0.5}] coordinates {(3,0) (6,0) (9,0) (12,0)};
\draw[-] (0,0.1)--(0,-0.1);
\draw[-] (15,0.1)--(15,-0.1);
\node at (0,-0.1) [below] {$0$};
\node at (1,0) [below] {$s_{0}^{1}$};
\node at (2,0) [below] {$s_{0}^{n_{0}}$};
\node at (3,0) [below] {\textcolor{red}{$s_1$}};
\node at (4,0) [below] {$s_1^{1}$};
\node at (5,0) [below] {$s_1^{n_1}$};
\node at (6,0) [below] {\textcolor{red}{$s_2$}};
\node at (9,0) [below] {\textcolor{red}{$s_{m-1}$}};
\node at (10,0) [below] {$s_{m-1}^{1}$};
\node at (11,0) [below] {$s_{m-1}^{n_{m-1}}$};
\node at (12,0) [below] {\textcolor{red}{$s_m$}};
\node at (13,0) [below] {$s_m^{1}$};
\node at (14,0) [below] {$s_m^{n_m}$};
\node at (15,-0.1) [below] {$2t$};
\draw[->] (5.77, 0.5) -- (9.23, 0.5);
\end{tikzpicture}
\end{equation*}
and for each $s \in \bs'$, $W(s)$ is defined by
\begin{equation*} 
W (s) := \left\{ \begin{array}{ll}
  W, & \text{if } s\in \bs'\setminus\bs , \\
  \mathring{W},
    & \text{if } s\in\bs.
\end{array} \right.
\end{equation*}
This means that time points in $\bs$ are treated specially in the integration.

Arguing as in Proposition \ref{prop:ac}, the above series is absolutely convergent.
Thus, we can simplify this multiple series by relabelling. Let $\bt=(t_n,\cdots,t_1)$ be a relabelling of $\bs'$, where $n=m+\sum_{i=0}^m n_i$. In each $\bt$, exactly $m$ elements are special, namely those originally labelled by $\bs = (s_m,\cdots,s_1)$. To formalize this, let $\mathcal{C}_n^m$ denote the set of $m$-element subsets of $\{1,\cdots,n\}$. Any $\bc\in \mathcal{C}_n^m$ can be written as $\bc=\{c_1,\cdots,c_m\}$, where $c_m>\cdots>c_1$. Let $\bt_{\bc}$ denotes $(t_{c_m},\cdots,t_{c_1})$. 
Accordingly, given $\bs$, there exists $\bc$ such that $\bs=\bt_{\bc}$.

 Plugging the above expansion of $\tr(\rho(0)\mathring{\mathcal{U}}(2t, \bs, 0))$ into the right-hand side of \eqref{eq:identity} and using the above relabelling, we arrive at
 \begin{align*}
        \text{RHS of \eqref{eq:identity}} &= \sum_{\substack{m=2\\[2pt] m \text{ is even}}}^{+\infty}  \sum_{n=m}^{+\infty}
        \sum_{\bc\in\mathcal{C}_n^m}\ii^n\int_{2t > t_n > \cdots > t_1 > 0} (-1)^{\#\{\bt < t\}} \sum_{\mf{q} \in \mQ(\bt_{\bc})} \mc{L}_{\Delta{B}}(\mf{q})\times\\
        &\qquad \times \tr (\rho(0)G^{(0)}(2t,t_n)W(\bc,n)\cdots W(\bc,1)G^{(0)}(t_1,0))
   \,\dd t_1 \cdots \,\dd t_n,
\end{align*}
where for each $n$, $\bc$, and $i\in \{1,\cdots,n\}$, $W (\bc, i)$ is defined by
\begin{equation*} 
W (\bc, i) := \left\{ \begin{array}{ll}
  W, & \text{if } i\not\in\bc , \\
  \mathring{W},
    & \text{if } i\in\bc.
\end{array} \right.
\end{equation*}

Now by Wick's condition \eqref{eq:Wick}, we have:
\begin{align*}
    &\tr (\rho(0)G^{(0)}(2t,t_n)W(\bc,n)\cdots W(\bc,1)G^{(0)}(t_1,0))\\
    = &\begin{cases}
        \tr_s(\rho_s\mc{U}_s (2t, \bt, 0))\sum_{\mf{q} \in \mQ(\bt\setminus \bt_{\bc})} \mathcal{L}_B(\mf{q}) & \text{if $n-m$ is even}, \\
         0 &\text{if $n-m$ is odd}.
    \end{cases}
\end{align*}

Therefore, 
\begin{align*}
    \text{RHS of \eqref{eq:identity}} =& \sum_{\substack{m=2\\[2pt] m \text{ is even}}}^{+\infty}  \sum_{\substack{n=m\\[2pt] n\text{ is even}}}^{+\infty} \sum_{\bc\in\mathcal{C}_n^m}\ii^n\int_{2t > t_n > \cdots > t_1 > 0} (-1)^{\#\{\bt < t\}}\tr_s(\rho_s\mc{U}_s (2t, \bt, 0))\times\\
    &\quad\times\left(\sum_{\mf{q} \in \mQ(\bt\setminus \bt_{\bc})} \mathcal{L}_B(\mf{q})\right)\left(\sum_{\mf{q} \in \mQ(\bt_{\bc})} \mc{L}_{\Delta{B}}(\mf{q})\right)\,\dd t_1 \cdots \,\dd t_n\\
    =&\sum_{\substack{n=2\\[2pt] n \text{ is even}}}^{+\infty}  \ii^n\int_{2t > t_n > \cdots > t_1 > 0} (-1)^{\#\{\bt < t\}}\tr_s(\rho_s\mc{U}_s (2t, \bt, 0))\times\\&\quad \times\sum_{\substack{m=2\\[2pt] m \text{ is even}}}^{n} \sum_{\bc\in\mathcal{C}_n^m}\left(\sum_{\mf{q} \in \mQ(\bt\setminus \bt_{\bc})} \mathcal{L}_B(\mf{q})\right)\left(\sum_{\mf{q} \in \mQ(\bt_{\bc})} \mc{L}_{\Delta{B}}(\mf{q})\right)\,\dd t_1 \cdots \,\dd t_n .
\end{align*}
Using the definition of $\mathcal{L}(\mf{q})$ in \eqref{eq:Lq} to expand $\mathcal{L}_{\wt{B}}(\mf{q})$ as a combination of $\mathcal{L}_{B}(\mf{q})$ and $\mathcal{L}_{\Delta B}(\mf{q})$, we have
\begin{equation*}
    \sum_{\mf{q} \in \mQ(\bt)}\mc{L}_{\wt{B}}(\mf{q})=\sum_{\substack{m=0\\[2pt] m \text{ is even}}}^{n} \sum_{\bc\in\mathcal{C}_n^m}\sum_{\mf{q}_1 \in \mQ(\bt\setminus \bt_{\bc})}\sum_{\mf{q}_2 \in \mQ(\bt_{\bc})} \mathcal{L}_B(\mf{q}_1) \mc{L}_{\Delta{B}}(\mf{q}_2).
\end{equation*}
Substituting this into the above calculation, we yield
\begin{align*}
    \text{RHS of \eqref{eq:identity}}
    =&\sum_{\substack{n=2\\[2pt] n \text{ is even}}}^{+\infty}  \ii^n\int_{2t > t_n > \cdots > t_1 > 0} (-1)^{\#\{\bt < t\}}\tr_s(\rho_s\mc{U}_s (2t, \bt, 0))\times\\&\quad \times\sum_{\substack{m=2\\[2pt] m \text{ is even}}}^{n} \sum_{\bc\in\mathcal{C}_n^m}\left(\sum_{\mf{q} \in \mQ(\bt\setminus \bt_{\bc})} \mathcal{L}_B(\mf{q})\right)\left(\sum_{\mf{q} \in \mQ(\bt_{\bc})} \mc{L}_{\Delta{B}}(\mf{q})\right)\,\dd t_1 \cdots \,\dd t_n \\
    =&\sum_{\substack{n=2\\[2pt] n \text{ is even}}}^{+\infty}  \ii^n\int_{2t > t_n > \cdots > t_1 > 0} (-1)^{\#\{\bt < t\}}\tr_s(\rho_s\mc{U}_s (2t, \bt, 0))\times\\&\quad \times\sum_{\mf{q} \in \mQ(\bt)} \left(\mc{L}_{\wt{B}}(\mf{q})-\mc{L}_{B}(\mf{q} )\right)\,\dd t_1 \cdots \,\dd t_n ,
\end{align*}
which is exactly the left-hand side of \eqref{eq:identity}, noting that $\mc{L}_{\wt{B}}(\emptyset)=\mc{L}_B(\emptyset)=1$. 
\end{proof}

To proceed, we also need the following integral identity which simplifies a particular high-dimensional integral to a power of a two-dimensional integral.
\begin{lemma}\label{lem:comb}
  Assume the function $B(\cdot,\cdot)$ is bounded. For $m$ even, the following equality holds:
  \begin{equation}
    \int_{\Sf > s_m > \cdots > s_1 > \Si}
    \sum_{\mf{q} \in \mQ(s_m, \cdots, s_1)} \mc{L}_B(\mf{q}) \,\dd s_1 \cdots \,\dd s_m = \frac{1}{(\frac{m}{2})!}\left(\int_{\Sf > s_2 > s_1 > \Si}B(s_1,s_2)\,\dd s_1 \dd s_2\right)^\frac{m}{2}.
  \end{equation}
\end{lemma}
\begin{proof}
Let $$1_{\{s_1>s_2\}}:=\left\{\begin{matrix}
 1, &\text{if } s_1>s_2,\\
 0, & \text{otherwise},
\end{matrix}\right.$$ be the indicator function.
  Let us unravel the right-hand side:
  \begin{align*}
    \left(\int_{\Sf > s_2 > s_1 > \Si}B(s_1,s_2)\,\dd s_1 \dd s_2\right)^\frac{m}{2} &=  \left(\int_{\Si}^{\Sf}\int_{\Si}^{\Sf}B(s^\downarrow,s^\uparrow)1_{\{s^\uparrow>s^\downarrow\}}\,\dd s^\downarrow \dd s^\uparrow\right)^\frac{m}{2}\\
    &= \int_{\Si}^{\Sf}\cdots\int_{\Si}^{\Sf}\prod_{i=1}^{\frac{m}{2}}\left(B(s^\downarrow_{i},s^\uparrow_{i})1_{\{s^\uparrow_{i}>s^\downarrow_{i}\}}\right)\,\prod_{i=1}^{\frac{m}{2}}\dd s^\downarrow_i\,\dd s^\uparrow_i\\
    &=\int_{\Si}^{\Sf}\cdots\int_{\Si}^{\Sf}\prod_{i=1}^{\frac{m}{2}}B(s^\downarrow_{i},s^\uparrow_{i})\prod_{i=1}^{\frac{m}{2}}1_{\{s^\uparrow_{i}>s^\downarrow_{i}\}}\,\prod_{i=1}^{\frac{m}{2}}\dd s^\downarrow_i\,\dd s^\uparrow_i,
  \end{align*}
  where in the first and second steps we simply rename the integration variables. We use $\uparrow$ and $\downarrow$ to denote the larger and smaller elements, respectively, within a pair.

  Next, we compute $\prod_{i=1}^{\frac{m}{2}}1_{\{s^\uparrow_{i}>s^\downarrow_{i}\}}$.
  Given $\frac{m}{2}$ pairs of real variables, each pair being ordered (\emph{i.e.}, $s_i^\downarrow<s_i^\uparrow$), our goal is to enumerate all possible total orderings of these $m$ variables.
  We begin by considering all possible ways to interlace these pairs.
  These interlacing configurations are precisely captured by the many-body diagrams $\mQ_m$ in \eqref{eq:all linking pair diagram example}, where each arc represents one pair of variables. 
  Then, for each interlacing configuration, there are $(\frac{m}{2})!$ ways to assign labels to the arcs, corresponding to all permutations of the $\frac{m}{2}$ pairs.
  The key observation is that this procedure exhaustively generates all possible total orderings of these $m$ variables. For illustration, when $m=4$, the interlacing configurations are
  \begin{equation*}
      \begin{tikzpicture}
\draw[-] (0,0)--(1.5,0);\draw plot[only marks,mark =*, mark options={color=black, scale=0.5}]coordinates {(0,0) (0.5,0) (1,0)(1.5,0)};
\draw[-] (0,0) to[bend left] (0.5,0);
\draw[-] (1,0) to[bend left] (1.5,0);
 \end{tikzpicture}
 ,\quad
 \begin{tikzpicture}
\draw[-] (0,0)--(1.5,0);\draw plot[only marks,mark =*, mark options={color=black, scale=0.5}]coordinates {(0,0) (0.5,0) (1,0)(1.5,0)};
\draw[-] (0,0) to[bend left] (1,0);
\draw[-] (0.5,0) to[bend left] (1.5,0);
 \end{tikzpicture}
 ,\quad
  \begin{tikzpicture}
\draw[-] (0,0)--(1.5,0);\draw plot[only marks,mark =*, mark options={color=black, scale=0.5}]coordinates {(0,0) (0.5,0) (1,0)(1.5,0)};
\draw[-] (0,0) to[bend left] (1.5,0);
\draw[-] (0.5,0) to[bend left] (1,0);
 \end{tikzpicture}.
  \end{equation*}
  We can name the nodes by $s_1^\downarrow,s_1^\uparrow,s_2^\downarrow,s_2^\uparrow$ respecting the order within each pair in the following ways
  \begin{equation*}
       \begin{tikzpicture}
\draw[-] (0,0)--(1.5,0);\draw plot[only marks,mark =*, mark options={color=black, scale=0.5}]coordinates {(0,0) (0.5,0) (1,0)(1.5,0)};
\draw[-] (0,0) to[bend left] (0.5,0);
\draw[-] (1,0) to[bend left] (1.5,0);
\node at (0,0) [below] {$s_1^\downarrow$};
\node at (1,0) [below] {$s_2^\downarrow$};
\node at (0.5,0) [below] {$s_1^\uparrow$};
\node at (1.5,0) [below] {$s_2^\uparrow$};
 \end{tikzpicture}
 ,\,
 \begin{tikzpicture}
\draw[-] (0,0)--(1.5,0);\draw plot[only marks,mark =*, mark options={color=black, scale=0.5}]coordinates {(0,0) (0.5,0) (1,0)(1.5,0)};
\draw[-] (0,0) to[bend left] (0.5,0);
\draw[-] (1,0) to[bend left] (1.5,0);
\node at (0,0) [below] {$s_2^\downarrow$};
\node at (1,0) [below] {$s_1^\downarrow$};
\node at (0.5,0) [below] {$s_2^\uparrow$};
\node at (1.5,0) [below] {$s_1^\uparrow$};
 \end{tikzpicture}
 ,\,
 \begin{tikzpicture}
\draw[-] (0,0)--(1.5,0);\draw plot[only marks,mark =*, mark options={color=black, scale=0.5}]coordinates {(0,0) (0.5,0) (1,0)(1.5,0)};
\draw[-] (0,0) to[bend left] (1,0);
\draw[-] (0.5,0) to[bend left] (1.5,0);
\node at (0,0) [below] {$s_1^\downarrow$};
\node at (0.5,0) [below] {$s_2^\downarrow$};
\node at (1,0) [below] {$s_1^\uparrow$};
\node at (1.5,0) [below] {$s_2^\uparrow$};
 \end{tikzpicture}
 ,\,
 \begin{tikzpicture}
\draw[-] (0,0)--(1.5,0);\draw plot[only marks,mark =*, mark options={color=black, scale=0.5}]coordinates {(0,0) (0.5,0) (1,0)(1.5,0)};
\draw[-] (0,0) to[bend left] (1,0);
\draw[-] (0.5,0) to[bend left] (1.5,0);
\node at (0,0) [below] {$s_2^\downarrow$};
\node at (0.5,0) [below] {$s_1^\downarrow$};
\node at (1,0) [below] {$s_2^\uparrow$};
\node at (1.5,0) [below] {$s_1^\uparrow$};
 \end{tikzpicture}
 ,\,
  \begin{tikzpicture}
\draw[-] (0,0)--(1.5,0);\draw plot[only marks,mark =*, mark options={color=black, scale=0.5}]coordinates {(0,0) (0.5,0) (1,0)(1.5,0)};
\draw[-] (0,0) to[bend left] (1.5,0);
\draw[-] (0.5,0) to[bend left] (1,0);
\node at (0,0) [below] {$s_1^\downarrow$};
\node at (0.5,0) [below] {$s_2^\downarrow$};
\node at (1.5,0) [below] {$s_1^\uparrow$};
\node at (1,0) [below] {$s_2^\uparrow$};
 \end{tikzpicture},\,
 \begin{tikzpicture}
\draw[-] (0,0)--(1.5,0);\draw plot[only marks,mark =*, mark options={color=black, scale=0.5}]coordinates {(0,0) (0.5,0) (1,0)(1.5,0)};
\draw[-] (0,0) to[bend left] (1.5,0);
\draw[-] (0.5,0) to[bend left] (1,0);
\node at (0,0) [below] {$s_2^\downarrow$};
\node at (0.5,0) [below] {$s_1^\downarrow$};
\node at (1.5,0) [below] {$s_2^\uparrow$};
\node at (1,0) [below] {$s_1^\uparrow$};
 \end{tikzpicture},
  \end{equation*}
  and this gives all total-orderings of $s_1^\downarrow,s_1^\uparrow,s_2^\downarrow,s_2^\uparrow$.
  
  In general, let $\mathcal{S}_{\frac{m}{2}}$ be the permutation group on $\{1,\cdots,\frac{m}{2}\}$. We have
  \begin{align*}
    \prod_{i=1}^{\frac{m}{2}}1_{\{s^\uparrow_{i}>s^\downarrow_{i}\}}&=\sum_{\pi\in\mathcal{S}_{\frac{m}{2}}}\sum_{\mf{q} \in \mQ_m} 1_{\mf{q}}(s_{\pi(1)}^\downarrow,s_{\pi(1)}^\uparrow,\cdots,s_{\pi(\frac{m}{2})}^\downarrow,s_{\pi(\frac{m}{2})}^\uparrow) ,
  \end{align*}
  where 
  \begin{equation*}
      1_{\mf{q}}(s_1^\downarrow,s_1^\uparrow,\cdots,s_{\frac{m}{2}}^\downarrow,s_{\frac{m}{2}}^\uparrow):=\left\{\begin{matrix}
 1, &\text{if } x_1<\cdots<x_m,\\
 0, & \text{otherwise},
\end{matrix}\right.
  \end{equation*}
    and the sequence $x_1,\cdots,x_m$ is a permutation of the inputs $s_1^\downarrow,s_1^\uparrow,\cdots,s_m^\downarrow,s_m^\uparrow$, determined by the pairing $\mf{q}$ according to the following algorithm: Given $\mf{q}$, we assign $s_1^\downarrow$ to the leftmost node, \emph{i.e.} set $x_1=s_1^\downarrow$, then we assign $s_1^\uparrow$ to the node connected to $x_1$ according to  $\mf{q}$. Next, we assign $s_2^\downarrow$ to the leftmost unassigned node, and then assign $s_2^\uparrow$ to its corresponding partner in $\mf{q}$. This procedure is repeated until all nodes are assigned.

  Note that $\prod_{i=1}^{\frac{m}{2}}B(s^\downarrow_{i},s^\uparrow_{i})$ is invariant when permuting the pairs, \emph{i.e.} for any permutation $\pi\in\mathcal{S}_{\frac{m}{2}}$, $\prod_{i=1}^{\frac{m}{2}}B(s^\downarrow_{i},s^\uparrow_{i})=\prod_{i=1}^{\frac{m}{2}}B(s^\downarrow_{\pi(i)},s^\uparrow_{\pi(i)})$.
  Therefore, 
  \begin{align*}
    \prod_{i=1}^{\frac{m}{2}}B(s^\downarrow_{i},s^\uparrow_{i})\prod_{i=1}^{\frac{m}{2}}1_{\{s^\uparrow_{i}>s^\downarrow_{i}\}}&= \prod_{i=1}^{\frac{m}{2}}B(s^\downarrow_{i},s^\uparrow_{i})\sum_{\pi\in\mathcal{S}_{\frac{m}{2}}}\sum_{\mf{q} \in \mQ_m} 1_{\mf{q}}(s_{\pi(1)}^\downarrow,s_{\pi(1)}^\uparrow,\cdots,s_{\pi(\frac{m}{2})}^\downarrow,s_{\pi(\frac{m}{2})}^\uparrow) \\
    &= \sum_{\pi\in\mathcal{S}_{\frac{m}{2}}} \prod_{i=1}^{\frac{m}{2}}B(s^\downarrow_{\pi(i)},s^\uparrow_{\pi(i)})\sum_{\mf{q} \in \mQ_m} 1_{\mf{q}}(s_{\pi(1)}^\downarrow,s_{\pi(1)}^\uparrow,\cdots,s_{\pi(\frac{m}{2})}^\downarrow,s_{\pi(\frac{m}{2})}^\uparrow).
  \end{align*}

  Now for every permutation $\pi\in\mathcal{S}_{\frac{m}{2}}$, we calculate
  \begin{align*}
    &\int_{\Si}^{\Sf}\cdots\int_{\Si}^{\Sf}\prod_{i=1}^{\frac{m}{2}}B(s^\downarrow_{\pi(i)},s^\uparrow_{\pi(i)})\sum_{\mf{q} \in \mQ_m} 1_{\mf{q}}(s_{\pi(1)}^\downarrow,s_{\pi(1)}^\uparrow,\cdots,s_{\pi(\frac{m}{2})}^\downarrow,s_{\pi(\frac{m}{2})}^\uparrow)\,\prod_{i=1}^{\frac{m}{2}}\dd s^\downarrow_i\,\dd s^\uparrow_i\\=& \sum_{\mf{q} \in \mQ_m}\int_{\Si}^{\Sf}\cdots\int_{\Si}^{\Sf}\prod_{i=1}^{\frac{m}{2}}B(s^\downarrow_{\pi(i)},s^\uparrow_{\pi(i)}) 1_{\mf{q}}(s_{\pi(1)}^\downarrow,s_{\pi(1)}^\uparrow,\cdots,s_{\pi(\frac{m}{2})}^\downarrow,s_{\pi(\frac{m}{2})}^\uparrow)\,\prod_{i=1}^{\frac{m}{2}}\dd s^\downarrow_i\,\dd s^\uparrow_i
    \\=& \sum_{\mf{q} \in \mQ(s_m, \cdots, s_1)}\int_{\Sf > s_m > \cdots > s_1 > \Si}
     \mc{L}_B(\mf{q}) \,\dd s_1 \cdots \,\dd s_m,
  \end{align*}
  where in the last step we simply rename the integration variables and use the definition of $\mc{L}_B(\mf{q})$ in \eqref{eq:Lq}. Notice that calculation result is identical for any permutation $\pi\in\mathcal{S}_{\frac{m}{2}}$.

As $|\mathcal{S}_{\frac{m}{2}}|=(\frac{m}{2})!$, we finally yield
\begin{align*}
    &\int_{\Si}^{\Sf}\cdots\int_{\Si}^{\Sf}\prod_{i=1}^{\frac{m}{2}}B(s^\downarrow_{i},s^\uparrow_{i})\prod_{i=1}^{\frac{m}{2}}1_{\{s^\uparrow_{i}>s^\downarrow_{i}\}}\,\prod_{i=1}^{\frac{m}{2}}\dd s^\downarrow_i\,\dd s^\uparrow_i\\
    =& \sum_{\pi\in\mathcal{S}_{\frac{m}{2}}} \int_{\Si}^{\Sf}\cdots\int_{\Si}^{\Sf}\prod_{i=1}^{\frac{m}{2}}B(s^\downarrow_{\pi(i)},s^\uparrow_{\pi(i)})\sum_{\mf{q} \in \mQ_m} 1_{\mf{q}}(s_{\pi(1)}^\downarrow,s_{\pi(1)}^\uparrow,\cdots,s_{\pi(\frac{m}{2})}^\downarrow,s_{\pi(\frac{m}{2})}^\uparrow)\,\prod_{i=1}^{\frac{m}{2}}\dd s^\downarrow_i\,\dd s^\uparrow_i\\
    =& \sum_{\pi\in\mathcal{S}_{\frac{m}{2}}}  \int_{\Sf > s_m > \cdots > s_1 > \Si}
    \sum_{\mf{q} \in \mQ(s_m, \cdots, s_1)} \mc{L}_B(\mf{q}) \,\dd s_1 \cdots \,\dd s_m\\
    =& \, (\frac{m}{2})! \int_{\Sf > s_m > \cdots > s_1 > \Si}
    \sum_{\mf{q} \in \mQ(s_m, \cdots, s_1)} \mc{L}_B(\mf{q}) \,\dd s_1 \cdots \,\dd s_m,
\end{align*}
which proves the desired result.
\end{proof}

Using the above two results, we can now prove the error bound \eqref{eq:bound}.
\begin{proof}[Proof of Theorem~\ref{thm:bound}]
By definition,
\begin{align*}
    &\Delta\langle O(t) \rangle = \langle  O(t)\rangle_{\wt{B}} - \langle  O(t)\rangle_{B} \\
    =&\sum_{\substack{m=0\\[2pt] m \text{ is even}}}^{+\infty}\ii^m
   \int_{2t > s_m > \cdots > s_1 > 0} (-1)^{\#\{\bs < t\}}\tr_s(\rho_s \mathcal{U}_s (2t, \bs, 0))
      \sum_{\mf{q} \in \mQ(\bs)} \left(\mc{L}_{\wt{B}}(\mf{q})-\mc{L}_{B}(\mf{q} )\right)
   \, \prod_{i=1}^{m}\dd s_i .
\end{align*}
Using Lemma \ref{lem:identity}, we obtain 
\begin{equation*}
    \Delta\langle O(t) \rangle = \sum_{\substack{m=2\\[2pt] m \text{ is even}}}^{+\infty} \ii^m \int_{2t > s_m > \cdots > s_1 > 0} (-1)^{\#\{\bs < t\}} \tr (\rho(0)\mathring{\mathcal{U} }(2t, \bs, 0))\sum_{\mf{q} \in \mQ(\bs)} \mc{L}_{\Delta{B}}(\mf{q})
   \,\prod_{i=1}^{m}\dd s_i.
\end{equation*}
Using the estimate
\begin{equation*}
    \left|\tr (\rho(0)\mathring{\mathcal{U} }(2t, \bs, 0))\right|\leqslant \|W_s\|^m\|O_s\|
\end{equation*}
to bound the above equation and using Lemma \ref{lem:comb} to simplify the integral, we obtain
\begin{align*}
    |\Delta\langle O(t) \rangle| &\leqslant \|O_s\|\sum_{\substack{m=2\\[2pt] m \text{ is even}}}^{+\infty}\|W_s\|^m
   \int_{2t > s_m > \cdots > s_1 > 0} 
      \sum_{\mf{q} \in \mQ(\bs)} \left| \mc{L}_{\Delta{B}}(\mf{q}) \right|
   \,\dd s_1 \cdots \,\dd s_m \\
   &=\|O_s\|\sum_{\substack{m=2\\[2pt] m \text{ is even}}}^{+\infty}\frac{1}{(\frac{m}{2})!}\|W_s\|^m\left(\int_{2t > s_2 > s_1 > 0}|\Delta  B(s_1,s_2)|\,\dd s_1 \dd s_2\right)^\frac{m}{2}\\
   &= \|O_s\| \left[\exp(\|W_s\|^2\int_0^{2t}\int_0^{s_2} |\Delta B(s_1,s_2)|\dd s_1\dd s_2)-1\right].\qedhere
\end{align*}
\end{proof}

\section{Example: Spin-Boson Model}\label{sec:sb}

Our result can be applied to many open quantum system models with Bosonic bath, including spin-boson model \cite{mascherpa2017open,Breuer2007}, Rabi model \cite{xie2017quantum}, Dicke model \cite{hepp1973superradiant}, and Anderson–Holstein model \cite{hewson2001numerical}. 
Let us look at the spin-boson model as an example, and see how it can be used to recover to the result in \cite{mascherpa2017open}.

\paragraph{Spin-Boson Model} To demonstrate the error bound in a specific model, we consider the spin-boson model in which the system is a single spin and the bath is given by a large number of harmonic oscillators. 
In detail, we have
\begin{displaymath}
\mc{H}_s = \lspan\{ \ket{1}, \ket{2} \}, \qquad
\mc{H}_b = \bigotimes_{l=1}^L \left( L^2(\mathbb{R}^3) \right),
\end{displaymath}
where $L$ is the number of harmonic oscillators. The corresponding Hamiltonians are
\begin{displaymath}
H_s = \epsilon \hat{\sigma}_z + \delta \hat{\sigma}_x, \qquad
H_b = \sum_{l=1}^L \frac{1}{2} (\hat{p}_l^2 + \omega_l^2 \hat{q}_l^2).
\end{displaymath}
The notations are described as follows:
\begin{itemize}
\item $\epsilon$: energy difference between two spin states.
\item $\delta$: frequency of the spin flipping.
\item $\hat{\sigma}_x, \hat{\sigma}_z$: Pauli matrices satisfying
  $\hat{\sigma}_x \ket{1} = \ket{2}$, $\hat{\sigma}_x \ket{2} = \ket{1}$,
  $\hat{\sigma}_z \ket{1} = \ket{1}$, $\hat{\sigma}_z \ket{2} = -\ket{2}$.
\item $\omega_l$: frequency of the $l$-th harmonic oscillator.
\item $\hat{q}_l$: position operator for the $l$-th harmonic oscillator defined by $\psi(q_1, \cdots, q_L) \mapsto q_l \psi(q_1, \cdots, q_L)$.
\item $\hat{p}_l$: momentum operator for the $l$-th harmonic oscillator defined by $\psi(q_1, \cdots, q_L) \mapsto -\ii \nabla_{q_l} \psi(q_1, \cdots, q_L)$.
\end{itemize}
The coupling between system and bath is assumed to be linear:
\begin{displaymath}
W = W_s \otimes W_b, \qquad
  W_s = \hat{\sigma}_z, \qquad W_b = \sum_{l=1}^L c_l \hat{q}_l,
\end{displaymath}
where $c_l$ is the coupling intensity between the $l$-th harmonic oscillator and the spin. 
We introduce the annihilation and creation operators:
\begin{displaymath}
\hat{a}_l = \sqrt{\frac{\omega_l}{2}}(\hat{q}_l+\frac{i}{\omega_l}\hat{p}_l),\qquad
 \hat{a}_l^\dagger = \sqrt{\frac{\omega_l}{2}}(\hat{q}_l-\frac{i}{\omega_l}\hat{p}_l),
\end{displaymath}
which satisfies the canonical commutation relation,
\begin{displaymath}
    [\hat a_i,\hat a_j^\dagger]=\delta_{ij}, \qquad [\hat a_i,\hat a_j]=0,\qquad [\hat a_i^\dagger ,\hat a_j^\dagger]=0.
\end{displaymath}
Leveraging the annihilation and creation operators, the bath Hamiltonian $H_b$ and the coupling $W_b$ can be rewritten as
\begin{displaymath}
H_b=\sum_{l=1}^L \omega_l\hat{a}_l^\dagger \hat{a}_l,\qquad
W_b=\sum_{l=1}^L \frac{c_l}{\sqrt{2\omega_l}}(\hat{a}_l+\hat{a}_l^\dagger).
\end{displaymath}
The spectral density is 
\[
J(\omega)=\sum_{l=1}^L \pi \frac{c_l}{\sqrt{2\omega_l}}\delta(\omega-\omega_i).
\]

Suppose the initial state of the bath is in the thermal equilibrium with inverse temperature $\beta$, \textit{i.e.} $\rho_b = Z_b^{-1} \exp(-\beta H_b)$, where $Z_b$ is a normalizing constant chosen such that $\tr(\rho_b) = 1$. 
As $H_b$ is quadratic and $\rho_b$ is thermal, Wick's condition \eqref{eq:Wick} holds with the unfolded bath correlation function \cite{Cai2020b}
\begin{equation} \label{eq:B}
B(\tau_1, \tau_2) =\sum_{l=1}^L \frac{c_l^2}{2\omega_l} \left[
  \coth \left( \frac{\beta \omega_l}{2} \right) \cos \omega_l (|\tau_1-t|-|\tau_2-t| )
  - \ii \sin \omega_l (|\tau_1-t|-|\tau_2-t| )
\right].
\end{equation}
When $\tau_1,\tau_2<t$, the unfolded bath correlation function reduces to the usual bath correlation function: 
\begin{equation*}
    B(\tau_1,\tau_2)=\sum_{l=1}^L \frac{c_l^2}{2\omega_l} \left[
  \coth \left( \frac{\beta \omega_l}{2} \right) \cos \omega_l (\tau_2-\tau_1)
  - \ii \sin \omega_l (\tau_2-\tau_1) \right].
\end{equation*}
Noting that for this particular bath correlation function we have by symmetry
\begin{equation*}
    \int_0^{2t}\int_0^{s_2} |\Delta B(s_1,s_2)|\dd s_1\dd s_2= 4\int_0^{t}\int_0^{s_2} |\Delta B(s_1,s_2)|\dd s_1\dd s_2,
\end{equation*}
and noting that $\|W_s\|=1$, we obtain
\begin{corollary}[Equation (6) in \cite{mascherpa2017open}]
\begin{equation*}
    |\Delta\langle O(t) \rangle|\leqslant \|O_s\|\left[\exp(4\int_0^{t}\int_0^{s_2}|\Delta B(s_1,s_2)|\dd s_1 \dd s_2)-1\right].
  \end{equation*}
\end{corollary}
This is exactly the main result in \cite{mascherpa2017open}, i.e., equation \eqref{eq:mascherpa bound} in the introduction.

\begin{remark}
    If we have an infinite number of harmonic oscillators, we can simply replace the summation with integration, leading to the following formalism \cite{mascherpa2017open}:
\begin{displaymath}
H_b=\int_0^\infty\omega\hat{a}_\omega^\dagger \hat{a}_\omega\,\dd \omega,\qquad
W_b=\int_0^\infty h(\omega)  (\hat{a}_\omega+\hat{a}_\omega^\dagger)\,\dd \omega.
\end{displaymath}
The spectral density is defined by 
\begin{displaymath}
    J(\omega)=\pi h(\omega)^2,
\end{displaymath}
and the corresponding unfolded bath correlation function is 
\begin{equation*}
\begin{aligned}
    B(\tau_1, \tau_2) &=\int_0^\infty\frac{J(\omega)}{\pi}  \bigg[
  \coth \left( \frac{\beta \omega_l}{2} \right) \cos \omega_l (|\tau_1-t|-|\tau_2-t| )\\
  &\qquad\qquad\qquad - \ii \sin \omega_l (|\tau_1-t|-|\tau_2-t| )
\bigg]\dd \omega.
\end{aligned}
\end{equation*}

\end{remark}

\section{Conclusion}\label{sec:conclusion}

We provide a rigorous proof for the error bounds of physical observable in an open quantum system due to the perturbation of bath correlation function. Our main result validates the estimate in \cite{mascherpa2017open} which relies on physical arguments.  

Such error bounds can be used to analyze various numerical methods for open quantum systems, including the hierarchical equation of motion (HEOM) \cite{Tanimura1989,cirio2025input} (for example, see Application in \cite{mascherpa2017open}), and pseudomode methods \cite{luo2023quantum,park2024quasi,tamascelli2018nonperturbative,mascherpa2020optimized}.
The error bound can also be used to analyze the approximation of spectral densities from experimental signatures \cite{pachon2014direct}. 

Another interesting direction is to apply the techniques developed in this work to the analysis of diagrammatic based numerical approaches for open quantum systems, such as the inchworm Monte Carlo method, trying to improve the error bound in \cite{Cai2020b}. 

The understanding of error caused by bath correlation function might open doors for the development of novel algorithms for open quantum systems. For instance, one potential idea is to use \eqref{eq:identity} to treat systems that are close to some reference system with existing numerical results. The identity \eqref{eq:identity} can then be used to calculate directly the difference between the two systems, where the series might converge faster than directly using the Dyson series. 

Finally, our result might be also applied to the learning of non-Markovian open quantum systems, as it quantifies the error of quantities of interest due to estimation errors of spectral density / bath correlation functions. We leave these for future works.

\section*{Acknowledgements}

This work was supported in part by National Science Foundation via grant DMS-2309378. 
The authors thank Zhenning Cai, Bowen Li, Lin Lin, Quanjun Lang, Yonggang Ren, and Siyao Yang for helpful discussions. 
Kaizhao Liu was partially supported by the elite undergraduate training program of School of Mathematical Sciences at Peking University.
This work was done during Kaizhao Liu's visit to Duke University. 
He thanks Rhodes Initiative of Informatics at Duke and the Mathematics Department for their hospitality.

\bibliographystyle{quantum}
\bibliography{boson}

\end{document}